\documentclass[a4paper,11pt]{article}
\usepackage{graphicx}
\usepackage[utf8x]{inputenc}
\usepackage{fullpage}
\usepackage{amsmath}
\usepackage{amssymb}
\usepackage{theorem}
\usepackage{stmaryrd}
\usepackage{hyperref}

\newtheorem{theorem}{Theorem}[section]
\newtheorem{lemma}[theorem]{Lemma}
\newtheorem{proposition}[theorem]{Proposition}

\theorembodyfont{\rm} 
\newtheorem{definition}[theorem]{Definition}
\newtheorem{remark}[theorem]{Remark}

\newenvironment{proof}{\noindent {\it Proof}.}{\hfill$\Box$}

\let\eps\varepsilon

\newcommand{\IN}{{\mathbb N}}
\newcommand{\IR}{{\mathbb R}}
\newcommand{\IQ}{{\mathbb Q}}
\newcommand{\IZ}{{\mathbb Z}}

\newcommand{\CC}{{\cal C}}
\newcommand{\CG}{{\cal G}}
\newcommand{\CI}{{\cal I}}
\newcommand{\CJ}{{\cal J}}
\newcommand{\CL}{{\cal L}}
\newcommand{\CP}{{\cal P}}

\newcommand{\Int}[1]{{\rm Int}\left(#1\right)}

\let\Lbrack\llbracket
\let\Rbrack\rrbracket

\DeclareMathOperator{\Rot}{Rot}

\DeclareMathOperator{\InfX}{\text{\upshape\sffamily\bfseries T}}
\newcommand{\InfG}{\InfX^{\circ}}

\newcommand{\sunny}[1][T]{\ensuremath{\mathcal{S}_1(#1)}}
\newcommand{\TR}{T_{\IR}}
\newcommand{\RotR}{\Rot_{\IR}}

\newcommand{\rhos}[1][]{\rho_{#1}}

\newcommand{\pluscover}[1]{\displaystyle \mathop{\longrightarrow}_{#1}^{+}}

\newcommand{\modi}{({\rm mod\ 1})}

\begin{document}

\title{Rotation set for maps of degree 1 on sun graphs}
\author{Sylvie Ruette}
\date{January 6, 2019}

\maketitle

\begin{abstract}
For a continuous map on a topological graph containing a unique loop
$S$, it is possible to define the degree and,  for a map of degree $1$,
rotation numbers.  It is known that  the set of rotation numbers of
points in $S$ is a compact interval  and for every rational $r$
in this interval there exists a periodic point of rotation number
$r$.  The whole rotation set (i.e. the set of all rotation
numbers) may not be connected and it is not known in general whether
it is closed.

A sun graph is the space consisting in finitely many segments attached by
one of their endpoints to a circle.  We show that, for a map of degree 1
on a sun graph,  the rotation set is closed and has finitely many
connected components. Moreover, for all but finitely many 
rational numbers $r$ in the
rotation set, there exists a  periodic point of rotation number $r$.
\end{abstract}

\section{Introduction}

In \cite{AR}, a rotation theory is developed for continuous self maps
of degree 1 of topological graphs having a unique loop.
A rotation theory is usually developed in the universal covering space by using
the liftings of the maps under consideration. The universal covering of a graph
containing a unique loop is an ``infinite tree modulo 1'' (see
Figure~\ref{fig:sungraph}). It turns out that the rotation
theory on the universal covering of a graph with a unique loop can be easily
extended to the setting of infinite graphs 
that look like the space $\widehat{G}$
on Figure~\ref{fig:hatG}. These spaces are defined in detail in
Section~\ref{subsec:liftedgraphs} 
and called \emph{lifted graphs}. Each lifted graph
$T$ has a subset $\widehat{T}$ homeomorphic to the real line $\IR$ that
corresponds to an ``unwinding'' of a distinguished loop of the original
graph. In the sequel, we identify $\widehat{T}$ with $\IR$.

\begin{figure}[ht]
\centerline{\includegraphics{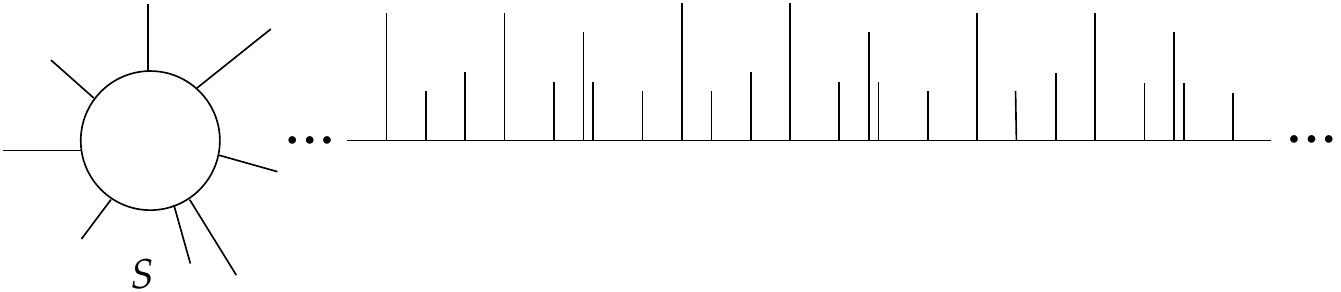}}
\caption{A sun graph on the left, and its universal covering
on the right.\label{fig:sungraph}}
\end{figure}

\begin{figure}[hbt]
\centerline{\includegraphics{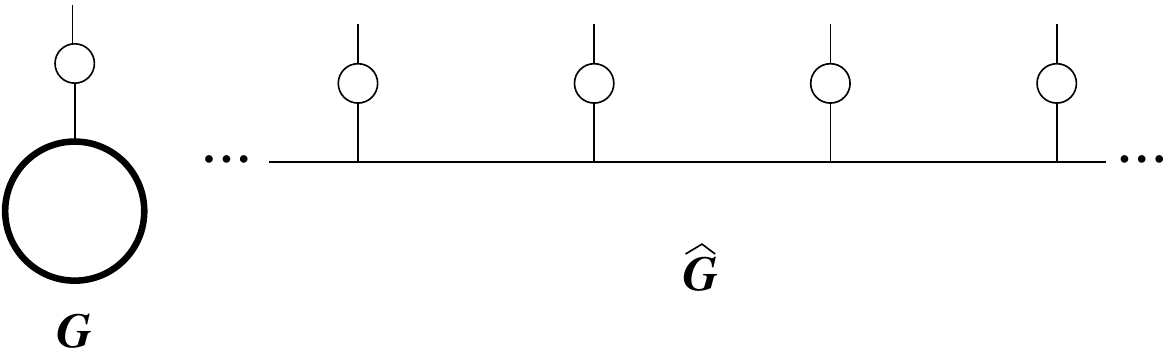}}
\caption{The graph $G$ is unwound with respect to the bold loop
to obtain $\widehat{G}$, which is a lifted graph. \label{fig:hatG}}
\end{figure}

Given a lifted graph $T$ and a map $F$ from $T$ to itself of degree one, there
is no difficulty to extend the definition of rotation number to this setting in
such a way that every periodic point has a rational rotation number as in
the circle case. However, the obtained rotation set $\Rot(F)$ may not be
connected (see \cite[Example~1.12]{AR}). Despite of this fact, it is proved in \cite{AR} that the
set $\Rot_{\IR}(F)$ corresponding to the rotation numbers of all points
belonging to $\IR$, has properties which are similar to (although weaker than)
those of the rotation interval for a circle map of degree one. Indeed, this set
is a compact non empty interval,
if $p/q\in \Rot_{\IR}(F)$ then there exists a periodic point of rotation number
$p/q$, and if $p/q\in \Int{\Rot_{\IR}(F)}$ then, for all large enough positive
integers $n$, there exists a periodic point of period $nq$ of rotation number
$p/q$.  

We conjecture that the whole rotation set $\Rot(F)$ is closed.
In this paper, we prove that, when the
space $T$ is the universal covering of a sun graph  (consisting in finitely
many disjoint segments attached by one of their endpoints to a circle,
see Figure~\ref{fig:sungraph}), then the rotation set is the union of 
finitely many compact intervals. Moreover, all but finitely many
rational points $r$ in $\Rot(F)$ are rotation numbers of periodic
{\modi} points. It turns out
that the proofs extend to a class of maps on graphs that we call
sun-like maps, which are defined in Section~\ref{subsec:hairymaps}.

This paper is the sequel of \cite{R9}, which deals with the graph $\sigma$,
i.e., a sun graph with a unique branch.
The results obtained for $\sigma$ in \cite{R9} are stronger but 
the methods cannot be generalised to sun graphs. Here the main tool
is the construction of
a countable oriented graph, and the symbolic dynamics on this graph
reflects enough of the dynamics of the original map to
compute rotation numbers and find periodic points. The idea is
inspired by the Markov diagram introduced by Hofbauer \cite{Hof} to study
piecewise monotone interval maps, although the goals are very different 
(the Markov diagram was used to study measures of maximal entropy).

The paper is organised as follows.
In Section 2, we give the definitions of the objects we deal with: 
lifted graphs, maps of degree 1, sun graphs and sun-like maps,
rotation numbers and rotation sets; we also recall the main results on
the rotation set of $\IR\subset T$ when $T$ is a lifted graph.
In Section 3, we recall the notion of positive covering,
which is a key tool to find periodic points.
In Section 4, we define a partition $\CP$ of the branches of $T$ 
(where $T$ is the universal covering of a sun graph)
according to  some dynamical properties 
and we state that the rotation number of a point can be computed using 
its itinerary according to the partition $\CP$.
Then, in Section 5, we define the covering graph $\CG$
associated to this partition $\CP$ ($\CG$ is a countable oriented graph), 
which gives a relation between 
itineraries of points and infinite paths in this graph, and we study the 
structure of the graph. Finally, in Section 6, we study the
rotation set of the covering graph
and, in Section 7, we pull back these results on the space
$T$ and we prove the main result about the rotation set of sun-like maps.

\section{Definitions and first properties}
\subsection{Lifted graphs}\label{subsec:liftedgraphs}
A \emph{topological finite graph} is a compact connected set $G$ containing a
finite subset $V$ such that each connected component of $G\setminus V$
is homeomorphic to an open interval. 
The aim of this section is to define in detail the class of \emph{lifted 
graphs} where we develop the rotation theory. They are obtained from 
a topological finite graph by unwinding one of its loops. This gives a new 
space that contains a subset homeomorphic to the real line and that is 
``invariant by a translation'' (see Figures \ref{fig:sungraph} and 
\ref{fig:hatG}). In \cite{AR}, a larger class of spaces called \emph{lifted
spaces} is defined.

\begin{definition}
Let $T$ be a connected topological space. We say that $T$ is a
\emph{lifted graph} if there 
exist a homeomorphism $h\colon\IR\to h(\IR)\subset T$,
and a homeomorphism $\tau\colon T\to T$ such that
\begin{enumerate}
\item $\tau(h(x)) = h(x+1)$ for all $x\in\IR$,
\item the closure of each connected component of  $T \setminus h(\IR)$ is a
topological finite graph that intersects $h(\IR)$ at a single point,
\item the number of connected components $C$ of $T \setminus h(\IR)$ such that
$\overline{C} \cap h([0,1])\neq \emptyset$ is finite.
\end{enumerate}
The class of all lifted graphs will be denoted by $\InfG$.
\end{definition}

To simplify the notation, in the rest of the paper we identify $h(\IR)$
with $\IR$ itself. In this setting, the map
$\tau$ can be interpreted as a translation by 1. So, for all $x \in T$, we
write $x+1$ to denote $\tau(x)$. Since $\tau$ is a homeomorphism, this notation
can be extended by denoting $\tau^m(x)$ by $x+m$ for all $m \in \IZ$.

We endow a lifted graph $T$ with 
a distance $d$ invariant by the translation $\tau$, i.e., 
$\forall x,y\in T$, $d(x+1,y+1)=d(x,y)$.

A loop is a subset homeomorphic to a circle.
If $G$ is a topological finite graph with a unique loop, 
then its universal covering is an infinite tree (i.e., it has no 
loop) and belongs to $\InfG$. Figure~\ref{fig:sungraph} illustrates this
situation. Because of (ii) in the previous definition, 
if the topological finite 
graph $G$ has several loops, the infinite graph obtained by unwinding a 
distinguished loop may or may not be a lifted graph. The essential property of
the class $\InfG$ is the existence of a natural retraction from $T$ to $\IR$.

\begin{definition}
Let $T\in \InfG$. The retraction $r_{\IR}\colon T\to\IR$ is the continuous map 
defined as follows. When $x\in\IR$, then $r_{\IR}(x) := x$. When
$x \notin \IR$, there exists a connected component $C$ of
$T \setminus \IR$ such that $x\in C$ and $\overline{C}$ intersects $\IR$ at a
single point $z$, and we let $r_{\IR}(x):=z$.
\end{definition}

\subsection{Maps of degree 1 and rotation numbers}

A standard approach to study the periodic points and orbits of a graph
map is to work at lifting level with the periodic {\modi} points.
The results on the lifted graph can obviously been pulled back to the original
graph (see \cite{AR}).
Moreover, the rotation numbers have a signification only for maps of degree 1,
as in the case of circle maps (see, e.g., \cite{ALM} for the rotation theory
for circle maps).
In this paper, we deal only with maps of degree 1 on lifted graphs.

\begin{definition}
Let $T \in \InfG$. A continuous map $F\colon T\to T$ is of degree $1$ if
$F(x+1)=F(x)+1$ for all $x\in T$.

A point $x\in T$ is called \emph{periodic} {\modi} for $F$ if there exists
a positive integer $n$ 
such that $F^n(x)\in x+\IZ$. The \emph{period} of $x$ is the
least integer $n$ satisfying this property.
\end{definition}

\begin{definition}
Let $T\in\InfG$, $F\colon T\to T$ a continuous map of degree 1
and $x\in T$. When the limit exists, the \emph{rotation number} of $x$
is
\[
\rho_F(x):=\lim\frac{r_{\IR}\circ F^n (x)-r_{\IR}(x)}{n}.
\]
\end{definition}

The next, easy lemma states that two points in the same orbit have the same
rotation number, as well as two point equal {\modi}.

\begin{lemma}\label{lem:first-properties}
Let $T\in\InfG$, $F\colon T\to T$ a continuous map of degree 1,
and $x\in T$ such that $\rho_F(x)$ exists.
\begin{enumerate}
\item $\forall k\in\IZ$, $\rho_F(x+k)=\rho_F(x)$.
\item $\forall n\geq 1$, $\rho_{F}(F^n(x))=\rho_F(x)$.
\end{enumerate}
\end{lemma}

\begin{remark}
If $F^q(x)=x+p$ with $q\in\IN$ and $p\in\IZ$, then $\rho_F(x)=p/q$. Therefore
all periodic {\modi} points have rational rotation numbers.
\end{remark}

An important object that synthesises the information about rotation numbers
is the \emph{rotation set}. 

\begin{definition}
Let $T\in\InfG$ and $F\colon T\to T$ a continuous map of degree 1.
For  $E\subset T$, the \emph{rotation set} of $E$ is:
\[
\Rot_E(F)   := \{\rho_F(x)\mid x\in E\text{ and }\rho_F(x)\text{ exists}\}.
\]
When $E=T$, we omit the subscript and we write 
$\Rot(F)$ instead of $\Rot_T(F)$.
\end{definition}

We define 
\[
\TR:=\overline{\bigcup_{n\geq 0}F^n(\IR)}.
\]
The next theorem summarises the properties of $\TR$ and $\Rot_{\IR}(F)$
(see Lemma 5.2 and Theorems 3.1, 5.7, 5.18 in \cite{AR}).

\begin{theorem}\label{theo:RotR}
Let $T\in \InfG$ and $F\colon T\to T$ a continuous map of degree $1$. 
Then $\TR\in \InfG$, $\Rot_{\TR}(F)=\Rot_{\IR}(F)$ and the set
$\Rot_{\IR}(F)$ is a non empty compact interval. Moreover, if $r\in 
\Rot_{\IR}(F)\cap\IQ$, then there exists a periodic {\modi} point
$x\in \TR$ such that $\rhos(x)=r$.
\end{theorem}

\subsection{Sun graphs and sun-like maps}\label{subsec:hairymaps}

A \emph{sun graph} is a topological finite graph that looks like 
the graph $S$ on Figure~\ref{fig:sungraph}.
It is composed of a circle and finitely many disjoint compact intervals,
each interval being attached by one of its endpoints to the circle. 

Let $T\in\InfG$ and $F\colon T\to T$ a continuous map of degree $1$.
We define 
\[
X:=\overline{T\setminus \TR}\cap r_{\IR}^{-1}([0,1)).
\]
Then $X$ is composed of finitely many finite graphs and $T=\TR\cup(X+\IZ)$.
Note that $\TR$ and $X$ implicitly depend on $F$.

\medskip
If $T$ is the lifting of a sun graph, then $X$ is either empty, or
composed of finitely many disjoint intervals, each intersecting $\TR$
at one of its endpoint. Maps with the same properties will be called
\emph{sun-like maps}.

\begin{definition}
Let $T\in\InfG$ and $F\colon T\to T$ a continuous map of degree $1$.
If $T\setminus \TR\cap r_{\IR}^{-1}([0,1))$ is composed of finitely many 
intervals whose closures are disjoint, we say that $F$ is a
\emph{sun-like} map. The intervals
\[\{\overline{C}\mid C \text{ is a connected component of }
T\setminus \TR\cap r_{\IR}^{-1}([0,1))\}\]
are called the \emph{branches} of $F$ and denoted by $(X^i)_{i\in \Lambda}$,
where $\Lambda$ is some finite set of indices.
The set of all sun-like maps of degree 1 on $T$ is denoted by $\sunny$.  
\end{definition}

\begin{remark}
In a sun graph, two different segments do not meet the circle at the same
point because the segments are compact and disjoint. 
Similarly, two branches of a sun-like map are not allowed to 
have a common endpoint in $\TR$. This property prevents $F$ from oscillating
infinitely many times between two branches. On the contrary, $F$ may
oscillate between a branch and $\TR$.
\end{remark}

\begin{definition}
Let $F$ be a sun-like map and $(X^i)_{i\in\Lambda}$ its branches.
Each branch $X^i$ may be endowed with two
opposite orders. We choose the one such that $\min X^i$ is the
one-point intersection $X^i\cap \TR$. 
\end{definition}

Consider a sun-like map $F$. Because of the 
definition of sun-like maps, all the paths starting in $\TR$ and ending in 
some branch $X^i$ must pass through the one-point intersection $\TR\cap X^i$. 
Thus, if $E$ is a connected set in $T$ containing one point of $\TR$ and 
one point of $X^i$, then $E$ contains $\TR\cap X^i=\{\min X^i\}$. 
We shall use this property several times.

\section{Positive covering}

\begin{definition}
Let $F$ be a sun-like map and $(X^i)_{i\in\Lambda}$ its branches.
For every $i\in\Lambda$, 
the retraction map $r_i\colon T\to X^i$ can be defined in a natural way 
by $r_i(x):=x$ if $x\in X^i$ and $r_i(x):=\min X^i$ otherwise.
\end{definition}

The notion of positive covering for subintervals of $\IR$
has been introduced in \cite{AR}. It can be extended for subintervals of
any subset of $T$ on which a retraction can be defined, as in
\cite{R9}. In this
paper, we shall use positive covering in the branches of $F$.
All properties of positive covering remain valid in this context.
In particular, if a compact interval $I$ positively $F$-covers itself, 
then $F$ has a fixed
point in $I$ (Proposition~\ref{prop:+cover-periodic}).

\begin{definition}
Let $T\in\InfG$, $F\in\sunny$ and $I,J$ two non empty compact
subintervals with $I\subset X^i$ and $J\subset X^j$, where $X^i,X^j$ are two
branches of $F$ ($X^i$ and $X^j$ may be equal). 
Let $n$ be a positive integer and $p\in\IZ$. 
We say that $I$ {\em positively $F^n$-covers} $J+p$
and we write $I\pluscover{F^n} J+p$ if there exist $x,y\in I$
with $x\leq y$ (with respect to the order in $X^i$) such that
$r_j(F^n(x)-p)\leq \min J$ and $\max J\leq r_j(F^n(y)-p)$ (with respect to
the order in $X^j$). In this situation, 
we also say that $I+q$ positively $F^n$-covers $J+p+q$ for all $q\in \IZ$.
\end{definition}

\begin{remark}
If $F^n(x)\in \TR$ and $J\subset X^j$ for some $j$, then the inequality 
$r_j(F^n(x)-p)\leq \min J$ is automatically satisfied. 
We shall often use this remark to prove that
an interval positively covers another.
\end{remark}

The next lemma is \cite[Lemma~2.2(c)]{AR}. It
states that positive coverings can be concatenated.

\begin{lemma}\label{lem:concatenatecovering}
Let $I, J, K$ be non empty compact intervals, each one included
in some branch of $F$. Let $n, m$ be positive integers and $p,q\in\IZ$. If
$I\pluscover{F^n}J+p$ and $J\pluscover{F^m}K+q$, then
$I\pluscover{F^{n+m}}K+p+q$.
\end{lemma}

The next proposition is \cite[Proposition 2.3]{AR}, rewritten in some less 
general form. 

\begin{proposition}\label{prop:+cover-periodic}
Let $T\in\InfG$ and $F\in \sunny$. Let $I_0,\ldots,I_{k-1}$ 
be non empty compact intervals, each one included in some branch of $F$, 
and $p_1,\ldots, p_k\in\IZ$. Suppose that we have
a chain of positive coverings:
\[
I_0\pluscover{F}I_1+p_1\pluscover{F}I_2+p_2
\pluscover{F}\cdots\qquad \cdots
I_{k-1}+p_{k-1}\pluscover{F} I_0+p_k.
\]
Then there exists a point 
$x_0\in I_0$ such that $F^{k}(x_0)=x_0+p_k$ and
$F^i(x_0)\in I_i+p_i$ for all $i\in\Lbrack 1, n-1\Rbrack$.
\end{proposition}

\section{Partition of the branches and itineraries}

Since $\TR$ is $F$-invariant, if $F^{n_0}(x)\in\TR$ for some $n_0\ge 0$, 
then $F^n(x)\in\TR$ for all $n\geq n_0$ and $\rho_F(x)\in\Rot_{\TR}(F)$. 
The properties of the rotation set $\Rot_{\TR}(F)$ has been
recalled in Theorem~\ref{theo:RotR}. 
Consequently, it remains to consider the points whose
orbits do not fall in $\TR$. Since $T=\TR\cup(X+\IZ)$ and $\rho_F(x+1)=
\rho_F(x)$ (Lemma~\ref{lem:first-properties}(i)), we have
\begin{equation}\label{eq:decomposition-RotF}
\Rot(F)=\Rot_{\TR}(F)\cup\Rot_{X^\infty}(F), \text{ where }
X^{\infty}:=\{x\in X\mid \forall n\geq 0, F^n(x)\in X+\IZ\}.
\end{equation}

Our first step consists in dividing each branch of $X$ according to
the location of the images in one of the sets $X^i+p, i\in\Lambda, p\in\IZ$.
 
\begin{lemma}\label{lem:basic-partition}
Let $T\in\InfG$, $F\in\sunny$ and $(X^i)_{i\in\Lambda}$ the branches of $F$. 
For every branch $X^i$, there exist an 
integer $N_i\geq 0$, disjoint non empty compact intervals 
$X_1^i,\ldots, X_{N_i}^i\subset X^i$
and, for every $j\in\Lbrack 1,N_i\Rbrack$, there exist
$\ell(X_j^i)\in \Lambda$ and $p(X_j^i)\in\IZ$ such that
\begin{enumerate}
\item 
$X_1^i<X_2^i<\cdots<X_{N_i}^i$ (with respect to the order in $X^i$),
\item
$F(X_j^i)\subset \left(X^{\ell(X_j^i)}+p(X_j^i)\right)\cup \Int{\TR}$,
\item
$F(\min X_j^i)=\min X^{\ell(X_j^i)}+p(X_j^i)\in\TR$,
\item 
$\displaystyle F\left(X\setminus \bigcup_{i\in \Lambda, j\in\Lbrack 1,N_i
\Rbrack}X_j^i\right)\cap (X+\IZ)=\emptyset$.
\end{enumerate}
\end{lemma}

\begin{proof}
We fix $i\in\Lambda$.
If $F(X^i)\cap(X+\IZ)=\emptyset$, we take $N_i=0$ and there is nothing to do.
Otherwise, we can define $a_1:=\min\{x\in X^i\mid F(x)\in X+\IZ\}$
and $p_1\in\IZ$ such that $F(a_1)\in X+p_1$. Since $F$ is a sun-like map,
there is a unique $\ell_1\in\Lambda$ such that $a_1\in X^{\ell_1}+p_1$.
We define
\[
b_1:=\max\{x\in [a_1,\max X^i]\mid F(x)\in X^{\ell_1}+p_1 \text{ and }
F([a_1,x])\subset (X^{\ell_1}+p_1)\cup \Int{\TR}\},
\]
and $X_1^i:=[a_1,b_1]$. 
Then $X_1^i$ satisfies (ii) with $p(X_1^i):=p_1$ and $\ell(X_1^i):=\ell_1$. 
Moreover, $F(\min X^i)\in\TR$
because $\min X^i\in\TR$, which implies that $F([\min X^i,a_1])$ contains
$\min X^{l_1}+p_1$ by connectedness. 
Thus $F(a_1)=\min X^{l_1}+p_1$ by minimality of 
$a_1$, which is (iii) for $X_1^i$.

We define $X_2^i, \ldots X_{N_i}^i$ inductively.
Suppose that $X_j^i=[a_j,b_j]$, $p_j=p(X_j ^i)$ and $\ell_j=\ell(X_j^i)$
are already defined and that $b_j$ satisfies: 
\[
b_j=\max\{x\in [a_j,\max X^i]\mid F(x)\in X^{\ell_j}+p_j\text{ and }
F([a_j,x])\subset (X^{\ell_j}+p_j)\cup \Int{\TR}\}.
\]
If $F((b_j,\max X^i])\subset \Int{\TR}$, we take $N_i:=j$ 
and the construction is
over. Otherwise, we define
\begin{equation}\label{eq:ai}
a_{j+1}:=\inf\{x\in (b_j,\max X^i]\mid F(x)\in X+\IZ\}.
\end{equation}

We first show that $a_{j+1}>b_j$.
By definition, there exists a sequence of points $(x_n)_{n\ge 0}$
in $(b_j,\max X^i]$ such that
$\lim_{n\to+\infty}x_n=a_{j+1}$ and $F(x_n)\in X+\IZ$
for all $n\ge 0$. Since the number of
branches is finite, we may assume (by taking a subsequence if necessary)
that there exists $\ell_{j+1}\in \Lambda$ such that  $F(x_n)\in 
X^{\ell_{j+1}}+\IZ$ for all $n\geq 0$.
Let $m_n\in\IZ$ be
such that $F(x_n)\in X^{\ell_{j+1}}+m_n$.
By continuity, $\lim_{n\to+\infty} r_{\IR}\circ F(x_n)=r_{\IR}\circ F(a_{j+1})$.
Since $r_{\IR}\circ F(x_n)=r_{\IR}(\min X^{\ell_{j+1}})+m_n$, 
this implies that the sequence of integers $(m_n)_{n\geq 0}$ is
ultimately constant, and equal to some integer $p_{j+1}$. Then $F(a_{j+1})
=\lim_{n\to+\infty}F(x_n)\in X^{\ell_{j+1}}+p_{j+1}$. 
By continuity, $F([a_{j+1}, x_n])\subset (X^{\ell_{j+1}}+p_{j+1})\cup 
\Int{\TR}$ for all large enough $n$.
Moreover, $F((b_j,a_{j+1}))\subset \Int{\TR}$ by definition of $a_{j+1}$.
If $p_{j+1}=p_j$ and $\ell_{j+1}=\ell_j$ then, for all large enough $n$,
we would have 
\[
F(x_n)\in X^{\ell_j}+p_j\text{ and }
F([b_j,x_n])\subset (X^{\ell_j}+p_j)\cup \Int{\TR},
\]
which would contradict the definition of $b_j$ because $x_n>b_j$. 
Hence $(p_{j+1},\ell_{j+1})\neq (p_j,\ell_j)$.
This implies that $a_{j+1}>b_j$ and, consequently, $a_{j+1}$ is actually
a minimum in \eqref{eq:ai}.
Since $F((b_i,a_{j+1}))$ is non empty and included in $\Int{\TR}$, necessarily 
$F(a_{j+1})$ is equal to $\min X^{\ell_{j+1}}+p_{j+1}$ by minimality of 
$a_{j+1}$.

Finally, we define
\begin{align*}
&b_{j+1}:=\\
&\max\{x\in [a_{j+1},\max X^i]\mid F(x)\in X^{\ell_{j+1}}+p_{j+1}
\text{ and }
F([a_{j+1},x])\subset (X^{\ell_{j+1}}+p_{j+1})\cup \Int{\TR}\},
\end{align*}
and $X^i_{j+1}:=[a_{j+1},b_{j+1}]$. Then
$X^i_{j+1}>X^i_j$, and (ii) and (iii) are satisfied with 
$p(X^i_{j+1}):=p_{j+1}$ and $\ell(X^i_{j+1}):=\ell_{j+1}$.

Let $\delta$ be the infimum of $d(x,y)$ where $x,y$ 
belong to two different sets of the form $X^\ell+p, \ell\in\Lambda,p\in\IZ$.
This is actually a minimum because the sets $X^\ell$ are compact,
$\Lambda$ is finite and the distance $d$
is invariant by translation by $1$. Moreover, 
$\delta>0$ because the branches are pairwise disjoint.

By uniform continuity of $F$ on the compact set $X$, 
there exists $\eta>0$ such that, if $x,y$ belong to $X$
with $d(x,y)<\eta$, then $|F(x)-F(y)|<\delta$. This
implies that $|a_{j+1}-b_j|\geq\eta$, otherwise $F(a_{j+1})$ and $F(b_j)$
would be in the same set $X^\ell+p$. This ensures that 
for a given $i\in\Lambda$, 
the number of intervals $X_j^i$ is finite, and the construction
ultimately ends. By construction, (iv) is satisfied.
\end{proof}

\begin{remark}
The fact that the sets $X^i_j$ are intervals
is very important because it will allow us to use positive coverings.
In an ideal situation, we would like to define the sets 
$(X^i_j)_{1\leq j\leq N_i}$ as the connected components of
$F^{-1}(X^i)\cap(X+\IZ)$. This is not possible in general because the number
of connected components may be infinite: this occurs when
$F$ oscillates infinitely many times between a branch and $\TR$.
\end{remark}

We call $\CP:=\{X_j^i\mid i\in\Lambda, j\in\Lbrack 1,N_i\Rbrack\}$
the \emph{basic partition} of $X$ (although the true partition of $X$
is $\CP\cup\{X\setminus \bigcup X_j^i\}$). The set
$X\setminus \bigcup X_j^i$, as well as $\TR$, plays the role of
``dustbin'': we do not need to care about points whose orbit falls in 
these sets because their rotation numbers belong to $\RotR(F)$.

According to Lemma~\ref{lem:basic-partition}(iv), every point
$x\in X$ such that $F(x)\in X+\IZ$ belongs to some $A\in\CP$, and this
$A$ is unique because the elements of $\CP$ are pairwise disjoint. This allows
us to code the orbits of the points of $X^\infty$ with respect to 
the partition $\CP$.

\begin{definition}
Let $X^\infty:=\{x\in X\mid \forall n\geq 0, F^n(x)\in X+\IZ\}$.
If $x\in X^\infty$ then, for every $n\geq 0$, there is a unique $A_n\in\CP$
such that $F^n(x)\in A_n+\IZ$. The sequence $(A_n)_{n\geq 0}$ is called
the \emph{itinerary} of $x$. Let $\Sigma\subset \CP^{\IZ^+}$ be the set of
all itineraries of points $x\in X^{\infty}$.
\end{definition}

The next lemma is straightforward.

\begin{lemma}\label{lem:itinerary-rho}
If $(A_n)_{n\geq 0}$ is the itinerary of $x\in X^\infty$, then
\[
\forall n\geq 0,\ F^n(x)\in A_n+p(A_0)+p(A_1)+\cdots+p(A_{n-1}),
\]
and, if $\rho_F(x)$ exists,
\[
\rho_F(x)=\lim_{n\to+\infty}\frac{p(A_0)+\cdots+p(A_{n-1})}{n}.
\]
\end{lemma}

\section{The covering graph associated to $\CP$}

Knowing the itinerary of a point $x\in X^\infty$ is enough to compute the 
rotation number of $x$. Therefore we can focus on the set $\Sigma$ of all 
itineraries. If
\[
\forall A,B\in\CP,\ F(A)\cap (B+p(A))\Longrightarrow F(A)\supset B+p(A),
\]
then it can be shown that $\Sigma$ is a Markov shift on the finite alphabet
$\CP$. In this case, the rotation set of $X^\infty$ can be easily computed
by the use of the Markov graph of $\Sigma$. In \cite{Zie} this is done
for transitive subshifts of finite type, and it is shown that in this case
the rotation set is a compact interval. When the Markov shift is not 
transitive, or equivalently when its Markov graph is not strongly connected
(see Definition~\ref{defi:connected} below), one has to look at the
different connected components of the graph, each of which giving an interval.

In general, $\Sigma$ may not be a Markov shift. We are going to build a
countable oriented graph, called the covering graph associated to $\CP$, 
that plays the 
role of the Markov graph: the symbolic itineraries can be read in the
graph and the structure of the graph (in particular its connected 
components) will give the structure of the rotation set.

\subsection{Definitions and first properties}

The construction of the covering graph
is inspired by the Markov diagram of a (non Markov) interval map, first
introduced by Hofbauer for piecewise monotone maps \cite{Hof}. Our
definition is closer to the Buzzi's version of the Markov diagram \cite{Buz},
although the basis of our covering
graph is always finite, as in Hofbauer's graph. In Hofbauer's and Buzzi's
constructions, the basis consists in monotone intervals, 
whereas our basis will be the basic partition $\CP$ (no monotonicity is
involved here).

\begin{definition}
If $A_i\in \CP$ for all $i\in\Lbrack 0,n\Rbrack$, we define
\[
\langle A_0A_1\ldots A_n\rangle:=F^n\left(\{x\in T\mid \forall i\in\Lbrack 0,
n\Rbrack,  F^i(x)\in A_i+\IZ\}\right)\cap X.
\]
\end{definition}

\begin{remark}
If the itinerary of $x$ begins with $A_0\ldots A_n$, then $F^n(x)\in 
\langle A_0\ldots A_n\rangle+\IZ$. In this sense, 
$\langle A_0\ldots A_n\rangle$ is the set of points whose ``past
itinerary'' is $A_0\ldots A_n$.
\end{remark}

The next lemma gives an alternative definition of
$\langle A_0\ldots A_n\rangle$ and states that this set is actually an
interval.

\begin{lemma}\label{lem:subintervalAn}
If $A_i\in \CP$ for all $i\in\Lbrack 0,n\Rbrack$, then
\begin{enumerate}
\item $\begin{array}[t]{rcl}
\langle A_0\ldots A_n\rangle
&=&F^n(A_0+\IZ)\cap F^{n-1}(A_1+\IZ)\cap\cdots \cap F(A_{n-1}+\IZ)\cap A_n\\
&=&F(\langle A_0\ldots A_{n-1}\rangle-p(A_{n-1}))\cap A_n
\end{array}$
\item
$\langle A_0A_1\ldots A_n\rangle$ is either empty, or a closed subinterval 
of $A_n$ containing $\min A_n$.
\end{enumerate}
\end{lemma}

\begin{proof} 
By definition,
\[
\langle A_0\ldots A_n\rangle =F^n\left((A_0+\IZ)\cap F^{-1}(A_1+\IZ)\cap\cdots
\cap F^{-n}(A_n+\IZ)\right)\cap X.
\]
Thus
\[
\langle A_0\ldots A_n\rangle =F^n(A_0+\IZ)\cap F^{n-1}(A_{1}+\IZ)\cap\cdots
\cap F(A_{n-1}+\IZ)\cap (A_n+\IZ)\cap X.
\]
Since $(A_n+\IZ)\cap X=A_n$, this gives the first equality of (i).
If we write this equality for $\langle A_0\ldots A_{n-1}\rangle$,
we see that $\langle A_0\ldots A_n\rangle =F(\langle A_0\ldots A_{n-1}\rangle
+\IZ)\cap A_n$. Since $\langle A_0\ldots A_{n-1}\rangle\subset A_{n-1}$, 
Lemma~\ref{lem:basic-partition}(ii) implies that
$F(\langle A_0\ldots A_{n-1}\rangle)\subset X+p(A_{n-1})$, and hence
$F(\langle A_0\ldots A_{n-1}\rangle+\IZ)\cap A_n=
F(\langle A_0\ldots A_{n-1}\rangle-p(A_{n-1}))\cap A_n$. This is the second
equality of (i).

\medskip
We show (ii) by induction on $n$. If $n=0$, then $\langle A_0\rangle=A_n$
and there is nothing to prove.

Suppose that $\langle A_0\ldots A_n\rangle \neq\emptyset$ and that
$\langle A_0\ldots A_{n-1}\rangle$ is a closed subinterval of $A_{n-1}$
containing $\min A_{n-1}$ (note that $\langle A_0\ldots A_{n-1}\rangle$
is not empty if $\langle A_0\ldots A_n\rangle\neq\emptyset$). By (i),
\[
\langle A_0\ldots A_n\rangle
=(F(\langle A_0\ldots A_{n-1}\rangle)-p(A_{n-1}))\cap A_n.
\] 
By continuity,
$F(\langle A_0\ldots A_{n-1}\rangle)$ is compact and connected, and thus
$(F(\langle A_0\ldots A_{n-1}\rangle)-p(A_{n-1}))\cap A_n$ is a closed
subinterval of $A_n$, which is non empty by assumption. Moreover,
$\langle A_0\ldots A_{n-1}\rangle$ contains $\min A_{n-1}$ by
the induction hypothesis, and $F(\min A_{n-1})\in\TR$ by 
Lemma~\ref{lem:basic-partition}(iii). This implies that the interval
$F(\langle A_0\ldots A_{n-1}\rangle)-p(A_{n-1})$ contains a point of
$\TR$ and a point of $A_n$, and thus it contains
$\min A_n$ by connectedness. Therefore 
(ii) holds for $\langle A_0\ldots A_n\rangle$.
This ends the induction.
\end{proof}

\medskip
We define an equivalence relation between the finite sequences of elements of
$\CP$.

\begin{definition}
Let $A_0,\ldots,A_n,B_0,\ldots,B_m\in\CP$. We set
$A_0\ldots A_n \sim B_0\ldots B_m$
if there is $k\in\Lbrack 0,\min(n,m)\Rbrack$ such that
\begin{equation}\label{eq:def-equiv}
\left\{\begin{array}[c]{l}
A_{n-i}=B_{m-i}\quad\text{for all}\quad i\in\Lbrack 0,k\Rbrack\\
\langle A_0\ldots A_{n-k}\rangle=A_{n-k}=B_{m-k}=
\langle B_0\ldots B_{m-k}\rangle
\end{array}\right.
\end{equation}
\end{definition}

\begin{remark}
It follows from
Lemma \ref{lem:subintervalAn}(i) that
\[
\text{if }A_0\ldots A_n \sim B_0\ldots B_m, \text{ then }
\langle A_0\ldots A_n\rangle=\langle B_0\ldots B_m\rangle.
\]
This means that, although the two sets come from different ``past 
itineraries'', their futures are indistinguishable.
If $\alpha=A_0\ldots A_n/\sim$ is an equivalence class, then 
$\langle \alpha \rangle$ denotes $\langle A_0\ldots A_n\rangle$, which is well
defined according to what precedes.
\end{remark}

The next result follows straightforwardly from Lemma~\ref{lem:subintervalAn}
and the fact that the elements of $\CP$ are disjoint.

\begin{lemma}\label{lem:alpha-subset-A}
If $\alpha=A_0\ldots A_n/\sim$ and $\langle\alpha\rangle\neq\emptyset$, 
then $A_n$ is the unique element $A\in\CP$
such that $\langle\alpha\rangle\subset A$.
\end{lemma}

Now we have all the notations to define the covering graph.

\begin{definition}
We define the oriented graph $\CG$ as follows:
\begin{itemize}
\item the set of vertices is the set of equivalence classes
$\alpha=A_0\ldots A_n/\sim$, where $n\geq 0$, $A_0,\ldots, A_n\in \CP$ and
$\langle A_0\ldots A_n\rangle\neq \emptyset$,
\item if $\alpha,\beta$ are two vertices, there is an arrow $\alpha\to\beta$
iff there exist $A_0,\ldots,A_n,A_{n+1}\in\CP$ such that
$\alpha=A_0\ldots A_n/\sim$ and $\beta=A_0\ldots A_nA_{n+1}/\sim$.
\end{itemize}
$\CG$ is called the \emph{covering graph} associated to $\CP$.
\end{definition}

The next lemma justifies the name ``covering graph''.

\begin{lemma}\label{lem:covering-in-graph}
If $\alpha\to \beta$ in $\CG$ and if $A\in\CP$ is such that 
$\langle\alpha\rangle\subset A$, then $\langle \alpha\rangle\pluscover{F}
\langle \beta\rangle+p(A)$.
\end{lemma}

\begin{proof}
Let $A_0,\ldots,A_n,A_{n+1}\in\CP$ be such that $\alpha=A_0\ldots A_n/\sim$
and $\beta=A_0\ldots A_nA_{n+1}/\sim$. Necessarily,
$A_n=A$ (Lemma~\ref{lem:alpha-subset-A}). 
According to Lemma~\ref{lem:subintervalAn}(ii), $\langle \alpha
\rangle$ is a closed subinterval of $A$ and $\min \langle \alpha
\rangle=\min A$. Thus Lemma~\ref{lem:basic-partition}(iii) implies that
\begin{equation}\label{eq:Fminalpha}
F(\min \langle \alpha\rangle)\in\TR.
\end{equation}
Moreover, $\langle \beta\rangle=F(\langle\alpha\rangle-p(A))\cap A_{n+1}$
by Lemma~\ref{lem:subintervalAn}(i). Thus
\begin{equation}\label{eq:maxFalpha}
\exists x_0\in \langle\alpha\rangle\text{ such that }
\max \langle \beta\rangle=F(x_0)-p(A).
\end{equation}
Since $\langle \beta\rangle$ is a closed subinterval of $A_{n+1}$, Equations
\eqref{eq:Fminalpha} and \eqref{eq:maxFalpha} imply that 
$\langle \alpha\rangle\pluscover{F} \langle \beta\rangle+p(A)$.
\end{proof}

\begin{definition}
The \emph{significant part} of $A_0\ldots A_n$ is $A_i\ldots A_n$, where
$i\in\Lbrack 0,n\Rbrack$ is the greatest integer such that
$A_0\ldots A_n\sim A_i\ldots A_n$.
If $\alpha$ is the equivalence class of $A_0\ldots A_n$, the significant
part of $\alpha$ is defined as the significant part of 
$A_0\ldots A_n$.
This does not depend on the representative of $\alpha$.

If $A_0\ldots A_n$ is the significant part of a vertex $\alpha$,
the \emph{height} of $\alpha$ is $H(\alpha)=n$. The \emph{basis} of $\CG$
is the set of vertices of height $0$, that is, $\{A/\sim\mid A\in \CP\}$.
We identify it with $\CP$.
\end{definition}

The next result, quite natural, will simplify the handling of arrows.

\begin{lemma}\label{lem:arrow-significantpart}
Let $\alpha=A_0\ldots A_n/\sim$ be a vertex of $\CG$ and
let $\alpha\to \beta$ be an arrow in $\CG$. Then there 
exists $A_{n+1}\in\CP$ such that
$\beta=A_0\ldots A_nA_{n+1}/\sim$.
\end{lemma}

\begin{proof}
By definition, there exist $B_0,\ldots, B_m, B_{m+1}\in\CP$ such that
$\alpha=B_0\ldots B_m/\sim$ and $\beta=B_0\ldots B_mB_{m+1}/\sim$.
Let $A_{n-k}\ldots A_0$ be the significant part of $\alpha$ (with 
$k=H(\alpha)\in\Lbrack 0, n\Rbrack$).
According to the definitions, we have $m\geq k$ and
\begin{gather*}
A_{n-i}=B_{m-i}\quad\text{for all }i\in\Lbrack 0,k\Rbrack,\\
\langle B_0\ldots B_{m-k}\rangle=B_{m-k}=A_{n-k}=
\langle A_0\ldots A_{n-k}\rangle.
\end{gather*}
This implies that $B_0\ldots B_mB_{m+1}\sim A_0\ldots A_nB_{m+1}$.
This proves the lemma with $A_{n+1}:=B_{m+1}$.
\end{proof}

We shall need some notions about paths in oriented graphs.

\begin{definition}
A \emph{(finite) path} in $\CG$ is a sequence of 
vertices $\alpha_0\ldots \alpha_n$ such that $\alpha_i\to \alpha_{i+1}$ 
is an arrow in $\CG$ for all $i\in\Lbrack 0,n-1\Rbrack$. A \emph{loop} is a 
path $\alpha_0\ldots \alpha_n$ such that $\alpha_n=\alpha_0$. An 
\emph{infinite path} in $\CG$ is an infinite sequence of
vertices $\bar\alpha=(\alpha_n)_{n\geq 0}$ such that 
$\alpha_i\to \alpha_{i+1}$ for all $i\geq 0$.

If $K$ is a subgraph of $\CG$, let $\Gamma(K)$ be the set of all
infinite paths in $K$.
\end{definition}

In the following, the infinite paths in $\CG$ will be denoted with a bar
(e.g. $\bar{\alpha}$) to distinguish them from vertices (e.g.
$\alpha_n,\beta$).

\begin{remark}
Endowed with the shift map $\sigma\colon (\alpha_n)_{n\geq 0}
\mapsto (\alpha_n)_{n\geq 1}$, $\Gamma(\CG)$ is a topological Markov chain
on a countable graph (see e.g. \cite{Ver1}).
We shall not explicitly use this structure of dynamical system, although
it will underlie the definition of rotation numbers of elements of
$\Gamma(\CG)$ and the relation between $\Gamma(\CG)$ and $X^\infty$.
\end{remark}

\subsection{Relation between itineraries of points of $X^\infty$ and infinite
paths in $\CG$}

The next result states that there is a correspondence between itineraries
of points of $X^\infty$ and infinite paths in $\CG$. This is a key 
property of the covering graph: it will allow us to pull back on $X^\infty$
the results obtained for $\CG$.

\begin{proposition}\label{prop:itinerary-G}
If $(A_n)_{n\geq 0}$ is the itinerary of some point $x\in X^\infty$ and
$\alpha_n=A_0\ldots A_n/\sim$, then $(\alpha_n)_{n\geq 0}$ is an infinite
path in $\CG$. Reciprocally, if  $(\alpha_n)_{n\geq 0}$ is an infinite path
in $\CG$ with $H(\alpha_0)=k$, then there exists a point $x\in X^\infty$
of itinerary $(A_n)_{n\geq 0}$ such that, for all $n\geq 0$,
$\alpha_n=A_0\ldots A_{n+k}/\sim$.
\end{proposition}

\begin{proof}
Suppose that $(A_n)_{n\geq 0}$ is the itinerary of $x\in X^\infty$. Then
$\langle A_0\ldots A_n\rangle$ contains $F^n(x)
-(p(A_0)+\cdots+p(A_{n-1}))$. Thus 
$\langle A_0\ldots A_n\rangle\neq\emptyset$ and
$\alpha_n=A_0\ldots A_n/\sim$ is a
vertex of $\CG$. It follows from the definition that $\alpha_n\to
\alpha_{n+1}$ is an arrow in $\CG$, that is, $(\alpha_n)_{n\geq 0}$ is an
infinite path in $\CG$.

Reciprocally, suppose that $(\alpha_n)_{n\geq 0}$ is an
infinite path in $\CG$. 
Let $A_0\ldots A_k$ be the significant part of $\alpha_0$, with
$k=H(\alpha_0)$.
According to Lemma~\ref{lem:arrow-significantpart}, we can
find inductively $A_{n+k}
\in\CP$ such that $\alpha_{n}=A_0\ldots A_{n+k}/\sim$ for all $n\geq 0$.
For every $n\ge 0$, let
\begin{align*}
E_n&:=\{x\in X\mid \forall i\in\Lbrack 0,n\Rbrack, F^i(x)\in A_i+\IZ\}\\
&\;=\{x\in X\mid \forall i\in\Lbrack 0,n\Rbrack, F^i(x)\in A_i+p(A_0)+\cdots+p(A_{i-1})\}
\end{align*}
Then $E_n$ is a compact set and $E_{n+1}\subset E_n$. Moreover,
$(F^n(E_n+\IZ))\cap X=\langle A_0\ldots A_n\rangle\neq\emptyset$
(Lemma~\ref{lem:subintervalAn}(i)). Hence
$E_n\neq\emptyset$. Therefore, the set $\bigcap_{n\geq 0}E_n$ is non empty
and every point $x$ in this set satisfies: $x\in X$ and $\forall n\geq 0,
F^n(x)\in A_i+\IZ$, that is, $x\in X^\infty$ and its itinerary is
$(A_n)_{n\geq 0}$.
\end{proof}

\subsection{Structure of the covering graph}

The oriented graph $\CG$ is usually infinite. However its infinite part 
is ``small'' and we shall exploit the particular structure of the covering
graph. Proposition~\ref{prop:structureG} gives the main properties of the
structure of $\CG$. It implies that ``most'' infinite paths come back
infinitely many times to the basis, which is rigorously stated in 
Proposition~\ref{prop:structure-paths}.

\begin{lemma}\label{lem:one-arrow}
Let $\alpha,\beta$ be vertices of $\CG$ such that there is an 
arrow $\alpha\to \beta$. Then there exist $A_0,\cdots, A_n, A_{n+1}\in\CP$
and $k\in\Lbrack 1, N_{\ell(A_n)}\Rbrack$ such that
$\alpha=A_0\ldots A_n/\sim$, $\beta=A_0\ldots A_n A_{n+1}/\sim$ and
$A_{n+1}=X^{\ell(A_n)}_k$. Moreover, for all $j\in\Lbrack 1, k-1\Rbrack$,
$\langle \alpha\rangle\pluscover{F} X^{\ell(A_n)}_i+p(A_n)$ and
$\alpha\to X^{\ell(A_n)}_i$ is an arrow in $\CG$.
\end{lemma}

\begin{proof}
We write $\alpha=A_0\ldots A_n/\sim$, $p:=p(A_n)$ and $\ell:=\ell(A_n)$.
Lemma~\ref{lem:arrow-significantpart} states that
there is $A_{n+1}\in\CP$ such that $\beta=A_0\ldots A_nA_{n+1}/\sim$.
The set $\langle A_0\ldots A_nA_{n+1}\rangle$ is non empty 
and satisfies
\begin{align*}
\langle A_0\ldots A_nA_{n+1}\rangle&=F(\langle A_0\ldots A_n\rangle-p)
\cap A_{n+1}\quad\text{(Lemma~\ref{lem:subintervalAn}(i))}\\
&\subset F(A_n-p)\cap A_{n+1}\\
&\subset X^{\ell}\cap A_{n+1}
\end{align*}
Thus $A_{n+1}$ is necessarily of the form $X^{\ell}_k$ for some
$k\in\Lbrack 1,N_{\ell}\Rbrack$.
According to Lemma~\ref{lem:covering-in-graph}, $\langle\alpha\rangle
\pluscover{F}\langle A_0\ldots A_n X^\ell_k\rangle+p$. Thus there
exists $x\in \langle\alpha\rangle$ such that
$F(x)=\min \langle A_0\ldots A_nX^\ell_k\rangle+p=\min X^\ell_k+p$
(the second equality comes from Lemma~\ref{lem:subintervalAn}(ii)).
Moreover, $\min \langle\alpha\rangle=\min A_n$ 
(by Lemma~\ref{lem:subintervalAn}(ii) again), and thus $F(\min \langle\alpha
\rangle)\in\TR$ by Lemma~\ref{lem:basic-partition}(iii).
This implies that $\langle\alpha\rangle\pluscover{F}[\min X^\ell,\min X^\ell_k]
+p$. Thus, for all $j\in \Lbrack 1,k-1\Rbrack$, we have
\[
\langle\alpha\rangle\pluscover{F} X^\ell_j+p,
\]
and hence $F(\langle\alpha\rangle+\IZ)\cap X^\ell_j=X^\ell_j$.
According to Lemma~\ref{lem:subintervalAn}(i),
$F(\langle\alpha\rangle+\IZ)\cap X^\ell_j=\langle 
A_0\ldots A_nX^\ell_j\rangle$.
Consequently, $A_0\ldots A_nX^\ell_j\sim X^\ell_j$ and
$\alpha\to X^\ell_j$ is an arrow in $\CG$.
\end{proof}

\begin{proposition}\label{prop:structureG}
Let $\alpha,\beta$ be vertices of $\CG$.
\begin{enumerate}
\item
All but at most one arrows starting from
$\alpha$ end at a vertex in the basis.
\item
If  $\alpha\to\beta$ with 
$H(\beta)>0$, then $H(\beta)=H(\alpha)+1$.
\end{enumerate}
\end{proposition}

\begin{proof}
We write $\alpha=A_0\ldots A_n/\sim$ and $p=p(A_n)$.

i) If $\alpha\to \beta$, 
Lemma~\ref{lem:one-arrow} states that there exists $k\in\Lbrack 1,
N_\ell\Rbrack$ such that $\beta=A_0\ldots A_n X^\ell_k/\sim$ and, for all
$j\in\Lbrack 1, k-1\Rbrack$, $\alpha\to X^\ell_j$ is an arrow in $\CG$.
This implies that there is at most one
vertex of the form $A_0\ldots A_n X^{\ell}_j/\sim$ of height different from $0$.
This proves (i).

\medskip
ii) Suppose that $\alpha\to\beta$ with $H(\beta)>0$. Let
$A_{n+1}\in\CP$ be such that $\beta=A_0\ldots A_nA_{n+1}/\sim$
(Lemma~\ref{lem:arrow-significantpart}). The significant
part of $\beta$ is $A_i\ldots A_nA_{n+1}$ with $i:=n+1-H(\beta)\leq n$
because $H(\beta)\geq 1$. Then by definition $\langle A_0\ldots A_i\rangle
=A_i$, and the significant part of $\alpha$ is $A_i\ldots A_n$.
Hence $H(\alpha)=n-i=H(\beta)-1$. This proves (ii).
\end{proof}

\begin{remark}\label{rem:structureG}
The structure of the graph $\CG$ can be deduced from 
Proposition~\ref{prop:structureG}: if we start from a vertex $A$ in the basis
and go up into the heights, there is a unique, finite or infinite, path
$(\alpha_n)$ starting at $A$ and such that $H(\alpha_i)=i$. Two such
paths starting at two different vertices in the basis are disjoint because
if $\alpha$ is a vertex of height $n>0$ and significant part $A_0\ldots A_n$
then $\alpha$ belongs only to the path starting at $A_0\in\CP$ (this paths
begins with $A_0, A_0A_1/\sim,\ldots, A_0\ldots A_n/\sim=\alpha$).
The only other arrows in $\CG$ end in the basis.
This is illustrated in Figure~\ref{fig:structureG}.

\begin{figure}[ht]
\centerline{\includegraphics{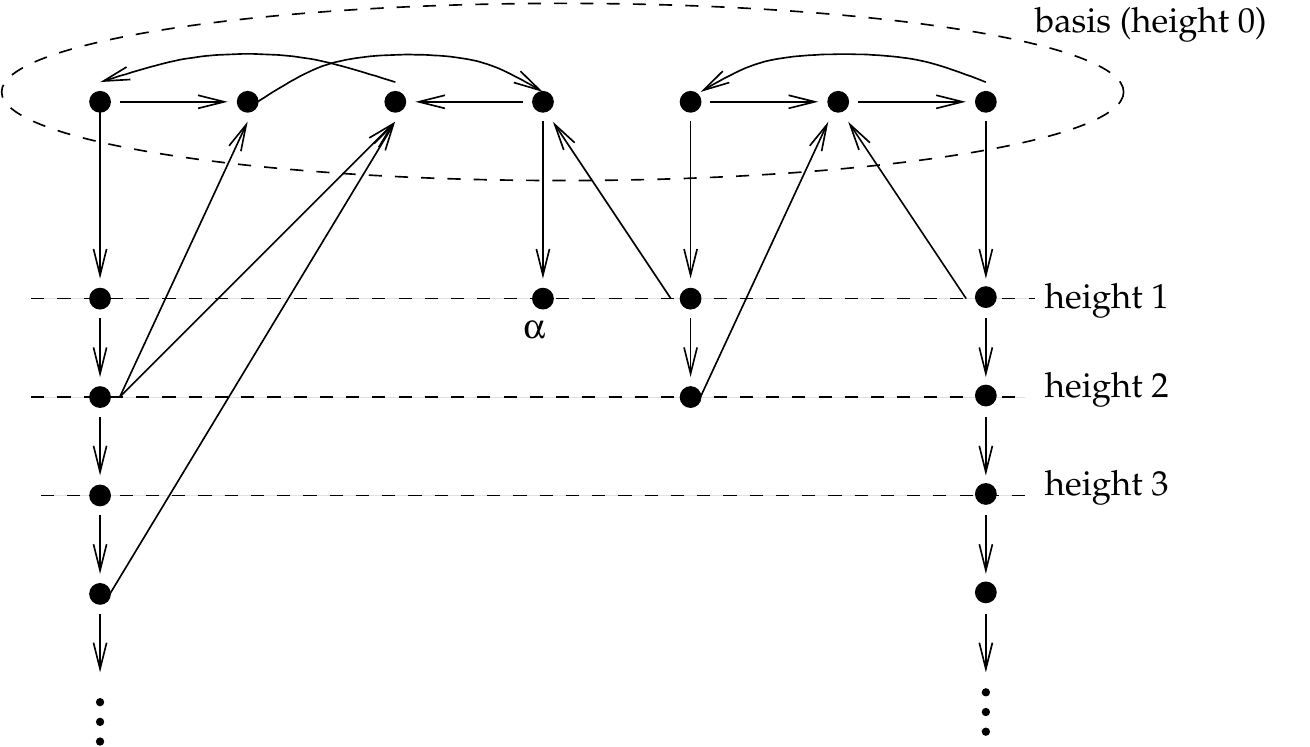}}
\caption{The structure of the covering graph. At the top: the finite basis.
\label{fig:structureG}}
\end{figure}
\end{remark}

\begin{definition}\label{defi:connected}
An oriented graph is \emph{strongly connected} if for every pair of vertices
$(u,v)$, there exists an oriented path of positive length
from $u$ to $v$. 

The \emph{connected components} of an oriented graph are the maximal
strongly connected subgraphs.
\end{definition}

\begin{remark}
Two connected components are either disjoint or equal.

Some vertices (called inessential vertices) 
may belong to no connected component, see e.g. the vertex $\alpha$ in the 
middle of Figure~\ref{fig:structureG}.
\end{remark}

\begin{proposition}\label{prop:structure-paths}
\begin{enumerate}
\item
Every connected component of $\CG$ meets the basis and  
the number of connected components is finite and bounded by $\#\CP$.
\item 
Let 
\[
\CI:=\{(\alpha_n)_{n\geq 0}\in\Gamma(\CG)\mid H(\alpha_0)=0\text{ and }
\forall n\geq 1, H(\alpha_n)>0\}.
\]
If $(\alpha_n)_{n\geq 0}\in\CI$, 
then $H(\alpha_n)=n$ for all $n\geq 0$. Moreover, $\CI$ is a finite set 
with $\#\CI\le \#\CP$.
\item  
If $(\alpha_n)_{n\geq 0}$ is an infinite path in $\CG$ then, either
there exists a connected component $C$ of $\CG$ such that all vertices
$\alpha_n$ belong to $C$ for all great enough $n$, or there exist
$(\beta_n)_{n\geq 0}\in\CI$ and $N,M\geq 0$ such that,
$\forall n\geq 0$, $\alpha_{N+n}=\beta_{M+n}$.
\end{enumerate}
\end{proposition}

\begin{proof}
i) If the vertex $v$ belongs to some connected component, there exists
a loop starting at $v$. According to the structure of $\CG$ (see 
Remark~\ref{rem:structureG}), every loop
goes through the basis, which is finite. Thus every connected component
meets the basis and the number of connected
components of $\CG$ is at most $\#\CP<+\infty$.

ii) Let $(\alpha_n)_{n\geq 0}\in \CI$. By Proposition~\ref{prop:structureG}(ii),
$H(\alpha_{n+1})=H(\alpha_n)+1$ for all $n\geq 0$, 
and thus $H(\alpha_n)=n$ for all $n\geq 0$.
According to Proposition~\ref{prop:structureG}(i), 
each vertex $\alpha_{n+1}$ is uniquely
determined by the properties that $\alpha_n\to\alpha_{n+1}$ and 
$H(\alpha_{n+1})\neq 0$. Since $\alpha_0\in\CP$,
the number of such infinite paths is less than or
equal to $\# \CP<+\infty$.

iii) Let $(\alpha_n)_{n\geq 0}$ be an infinite path in $\CG$. There are
two cases. 

\noindent $\bullet$ 
Suppose that there exists $N$ such that
$\forall n\geq N, \alpha_n\not\in\CP$. Then, by 
Proposition~\ref{prop:structureG}(ii),
$\forall n\geq N, H(\alpha_{n+1})=H(\alpha_n)+1$, and thus
$\forall n\geq 0$, $H(\alpha_{N+n})=H(\alpha_N)+n$. Let $A_0\ldots A_M$ be
the significant part of $\alpha_N$, with $M=H(\alpha_N)$. By 
Lemma~\ref{lem:arrow-significantpart}, there exists a sequence 
$(A_n)_{n\geq M+1}$ of elements of $\CP$ such that, for all $n\geq 0$,
$\alpha_{N+n}=A_0\ldots A_{M+n}/\sim$. We define $\beta_n:=A_0\ldots A_n/\sim$
for all $n\geq 0$. Since $H(\beta_M)=H(\alpha_N)=M$, we have $H(\beta_n)=n$
for all $n\in\Lbrack 0,M\Rbrack$, and $H(\beta_{M+n})=H(\alpha_{N+n})=M+n$ for all
$n\geq 0$. Hence $(\beta_n)_{n\geq 0}\in\CI$ and
$\alpha_{N+n}=\beta_{M+n}$ for all $n\geq 0$.

\noindent $\bullet$ 
Otherwise, there exist
infinitely many $n$ such that $\alpha_n\in\CP$. Since $\CP$ is finite,
there exists $A\in \CP$ such that $\alpha_n=A$ for infinitely many $n$.
Consequently, $A$ belongs to some connected component $C$ and
$\alpha_n$ belongs to $C$ for all great enough $n$.
\end{proof}

\section{The rotation set of the covering graph}
\subsection{Rotation numbers of infinite paths}

\begin{definition}
Let $\alpha\to \beta$ be an arrow in $\CG$ and let $A$ be the 
unique element of $\CP$ such that $\langle\alpha\rangle\subset
A\in\CP$. The \emph{weight} of the arrow $\alpha\to \beta$
is defined as $W(\alpha\beta):=p(A)$. 
\end{definition}

This naturally leads to the
following definition of rotation numbers for infinite paths
(see \cite{Zie} for a similar notion in the case of subshifts of finite type).

\begin{definition}
If $\alpha_0\ldots \alpha_n$ is a finite path in $\CG$, its \emph{length} is
$L(\alpha_0\ldots \alpha_n):=n$ and its \emph{weight} is 
$W(\alpha_0\ldots \alpha_n):=\sum_{i=0}^{n-1}W(\alpha_i\alpha_{i+1})$.

If $\bar\alpha=(\alpha_n)_{n\geq 0}\in \Gamma(\CG)$, then its \emph{rotation
number} is
\[
\rhos(\bar\alpha):=\lim_{n\to+\infty}\frac{W(\alpha_0\ldots \alpha_n)}{n}
\]
when this limit exists.

If $\Gamma'\subset \Gamma(\CG)$, let $\Rot_{\Gamma'}:=\{\rho(\bar\alpha)\mid
\bar\alpha\in\Gamma'\text{ and }\rho(\bar\alpha)\text{ exists}\}$.
\end{definition}

\begin{definition}
If $\gamma=\alpha_0\ldots\alpha_n$ and $\gamma'=\beta_0\ldots\beta_m$ are
two paths in $\CG$ with $\beta_0=\alpha_n$, let $\gamma\cdot\gamma'$ denote the
concatenation of the two paths, that is,
\[\gamma\cdot\gamma':=\alpha_0\ldots\alpha_n\beta_1\ldots\beta_m.
\]
If $\gamma$ is a loop then $\gamma^n:=\gamma\cdot\gamma\ldots\gamma$ 
is the $n$-time
concatenation of $\gamma$. We define similarly the concatenation of infinitely
many finite paths, and the concatenation of a finite path with an infinite
path; in these two cases, the resulting paths are infinite.

If $\gamma$ is a loop in $\CG$, let $\widetilde\gamma:=\gamma^{\infty}$
be the corresponding periodic infinite path.
\end{definition}

The next lemma is straightforward.

\begin{lemma}
\begin{enumerate}
\item If $\gamma,\gamma'$ are two finite paths that can be concatenated, then
$L(\gamma\cdot\gamma')=L(\gamma)+L(\gamma')$ and 
$W(\gamma\cdot\gamma')=W(\gamma)+
W(\gamma')$.
\item
 If $\gamma$ is a loop, then
$\rho(\widetilde \gamma)=\frac{W(\gamma)}{L(\gamma)}\in\IQ$.
\end{enumerate}
\end{lemma}

\subsection{The rotation set of a connected component}

We have seen in Proposition~\ref{prop:structure-paths} that an infinite path
either ultimately belongs to some connected component of $\CG$, or
ultimately coincide with some infinite path belonging to some finite set 
$\CI$. In this subsection, we focus on the first case and we study the
rotation set of a given connected component of $\CG$.

We start with a lemma that uses the concatenation of loops to get
rotation numbers.

\begin{lemma}\label{lem:interval}
\begin{enumerate}
\item Let $\gamma_1,\gamma_2$ be two loops in $\CG$ starting at the same
vertex $\alpha$. If $r\in [\rhos(\widetilde\gamma_1), 
\rhos(\widetilde\gamma_2)]\cap \IQ$, there exists a loop $\gamma$ starting
at $\alpha$ such that $\rhos(\widetilde\gamma)=r$.
\item Let $\alpha$ be a vertex of $\CG$ and, for every $n\geq 0$, 
let $\gamma_n$ be a loop in $\CG$ starting at
$\alpha$. If $\lim_{n\to+\infty}\rhos
(\widetilde\gamma_n)=s\in\IR$, then there exists $\bar\alpha=(\alpha_n)_{n\geq 0}
\in \Gamma(\CG)$ such that $\alpha_0=\alpha$, $\alpha_n=\alpha$ for
infinitely many $n$ and $\rhos(\bar\alpha)=s$.
\end{enumerate}
\end{lemma}

\begin{proof}
i) We write $r=p/q$ with $p\in\IZ$ and $q\in\IN$. 
Let $L$ be a multiple of $L(\gamma_1), L(\gamma_2)$ 
and $q$,
and let $k_1,k_2,k$ be such that $L=k_1L(\gamma_1)=k_2L(\gamma_2)=kq$.
We set $\gamma_i':=(\gamma_i)^{k_i}$ and $W_i:=W(\gamma_i')$ for 
$i\in\{1,2\}$, and
 $p':=kp$. In this way, $L(\gamma_i')=L$ and
$\rhos(\widetilde\gamma_i)=\rhos(\widetilde\gamma'_i)=W_i/L$
for $i\in\{1,2\}$, and $r=p'/L$. Since
$\rhos(\widetilde\gamma_1)\leq r\leq \rhos(\widetilde\gamma_2)$, we
have $W_1\leq p'\leq W_2$. We define
$\gamma:=(\gamma_1')^{W_2-p'}\cdot(\gamma_2')^{p'-W_1}$. This is a loop starting
at $\alpha$, and 
\[
\rhos(\widetilde\gamma)=\frac{W(\gamma)}{L(\gamma)}=
\frac{W_1(W_2-p')+W_2(p'-W_1)}{L(W_2-p')+L(p'-W_1)}
=\frac{(W_2-W_1)p'}{(W_2-W_1)L}=r.
\]
This proves (i).

ii) Let $(i_n)_{n\geq 0}$ be a sequence of positive integers and 
$\bar\alpha:=\gamma_0^{i_0}\cdot\gamma_1^{i_1}\cdot\ldots \gamma_n^{i_n}\ldots$.
This is an infinite path starting at $\alpha$ and passing at $\alpha$
infinitely many times. It can be shown that, if the sequence $(i_n)_{n\ge 0}$ 
increases
sufficiently fast, then $\rhos(\bar\alpha)=\lim_{n\to+\infty}
\rhos(\widetilde\gamma_n)=s$ (see e.g. the proof of \cite[Theorem~3.1]{AR} 
for a similar proof expliciting the growth of $(i_n)_{n\geq 0}$).
\end{proof}

\begin{proposition}\label{prop:RotC}
Let $C$ be a connected component of $\CG$. Then
\begin{enumerate}
\item
$\Rot_{\Gamma(C)}$ is a non empty compact interval.
\item
$\forall r\in \Int{\Rot_{\Gamma(C)}}\cap\IQ$, there exists
a loop $\gamma$ in $C$ such that $\rhos(\widetilde\gamma)=r$.
\end{enumerate}
\end{proposition}

\begin{proof}
The set $C\cap \CP$ is non empty by 
Proposition~\ref{prop:structure-paths}(i), and thus we can
fix $A\in C\cap \CP$. 
For every $B,B'\in C\cap \CP$, we choose a path
$u(B,B')$ from $B$ to $B'$. Since $\CP$ is finite, we can bound $L(u(B,B'))$
and $W(u(B,B'))$ by some quantities $L_0$ and $W_0$ respectively, 
independently of $B,B'$.
Let $\CL_A$ be the set of all loops starting at $A$. If $\gamma\in \CL_A$, then
the periodic path $\widetilde\gamma$ belongs to $\Gamma(C)$.

Let $\bar\alpha=(\alpha_n)_{n\geq 0}\in\Gamma(C)$ 
such that $\rhos(\bar\alpha)$ exists.
We are going to show that 
\begin{equation}\label{eq:density}
\forall \eps>0, \exists \gamma\in\CL_A\text{ such that }
|\rhos(\bar\alpha)-\rhos(\widetilde\gamma)|<\eps.
\end{equation}
If $u$ is a path from $A$ to $\alpha_0$ then 
$u\cdot (\alpha_n)_{n\geq 0}\in\Gamma(C)$ and it 
has the same rotation number as $\bar\alpha$.
Thus we can assume that $\alpha_0=A$.

We first show that there exist a sequence of integers
$(n_i)_{i\geq 0}$ increasing to infinity,
and finite paths $v_i$ from $\alpha_{n_i}$ to $A$ such
that $\forall i\geq 0, W(v_i)\leq W$ and $L(v_i)\leq L$, where
\[
W:=W_0+\max\{p(B)\mid B\in \CP\}\quad\text{and}\quad L:=L_0+1.
\]
Since $\bar\alpha$ is in $\Gamma(C)$, for all integers 
$N\ge 0$, there exists a path in $C$ from $\alpha_N$
to $A$. Because of the structure of $\CG$ (see 
Proposition~\ref{prop:structureG}), 
this implies that there exists $n\geq N$ and $B\in C\cap \CP$ 
such that $\alpha_n\to B$. Thus we can find a sequence
$(n_i)_{i\geq 0}$ increasing to infinity
and vertices $(B_i)_{i\geq 0}$ in $C\cap \CP$ such that
$\alpha_{n_i}\to B_i$ for all $i\geq 0$. We set
$v_i:=\alpha_{n_i}B_i\cdot u(B_i,A)$.
Then $L(v_i)\leq L$ and $W(v_i)\leq W$.

Now we define $\gamma_i\in\CL_A$ by concatenating $\alpha_0\ldots \alpha_{n_i}$
with $v_i$. Fix $\eps>0$ and let $i$ be great enough such that
\[
\left|\rhos(\bar\alpha)-\frac{W(\alpha_0\ldots\alpha_{n_i})}{n_i}
\right|<\frac{\eps}{3},\qquad \frac{|W|}{n_i}<\frac{\eps}{3},\quad\text{and}
\quad
\left(|\rhos(\bar\alpha)|+\eps/3\right)\frac{L}{n_i}<\frac{\eps}{3}.
\]
Since $\rhos(\widetilde\gamma_i)=\frac{W(\alpha_0\ldots\alpha_{n_i})+
W(v_i)}{n_i+L(v_i)}$, we have
\begin{align*}
|\rhos(\bar\alpha)-\rhos(\widetilde\gamma_i)|
&\leq \left|\rhos(\bar\alpha)-\frac{W(\alpha_0\ldots\alpha_{n_i})}{n_i}
\right|
+\left|\frac{W(\alpha_0\ldots\alpha_{n_i})+W(v_i)}{n_i+L(v_i)}
-\frac{W(\alpha_0\ldots\alpha_{n_i})}{n_i}\right|\\
&< \frac{\eps}{3}+\left|\frac{W(v_i)}{n_i+L(v_i)}-
W(\alpha_0\ldots\alpha_{n_i})\frac{ L(v_i)}{n_i(n_i+L(v_i))}\right|\\
&<\frac{\eps}{3}+\frac{|W|}{n_i}+\frac{|W(\alpha_0\ldots\alpha_{n_i})|}{n_i}
\frac{L}{n_i}\\
&<\frac{\eps}{3}+\frac{\eps}{3}+\left(|\rhos(\bar\alpha)|+\eps/3\right)\frac{L}{n_i}\\
&<\eps
\end{align*}
This proves Equation~\eqref{eq:density}. In other words, 
$\{\rhos(\widetilde\gamma)\mid \gamma\in\CL_A\}$ is dense in 
$\Rot_{\Gamma(C)}$.

The set $\Rot_{\Gamma(C)}$ is non empty because
there exists a loop $\gamma$ starting at $A\in C\cap\CP\neq\emptyset$,
and $\rhos(\widetilde\gamma)$ exists.
We set $a:=\inf\Rot_{\Gamma(C)}$ and $b:=\sup \Rot_{\Gamma(C)}$.
We suppose $a<b$, otherwise there is nothing to prove.
Let $r\in (a,b)\cap\IQ$ and let $\eps>0$ be such that $a+2\eps<r<b-2\eps$. 
Let $\bar\alpha,\bar\beta\in \Gamma(C)$ be such that $|\rhos(\bar\alpha)-a|
<\eps$ and $|\rhos(\bar\beta)-b|<\eps$.
By  Equation~\eqref{eq:density}, there exist
$\gamma_1,\gamma_2\in \CL_A$ such that 
$|\rhos(\widetilde\gamma_1)-\rhos(\bar\alpha)|<\eps$
and $|\rhos(\widetilde\gamma_2)-\rhos(\bar\beta)|<\eps$, and hence 
\[
\rhos(\widetilde\gamma_1)<r<\rhos(\widetilde\gamma_2).
\]
Then by Lemma~\ref{lem:interval}(i)
there exists $\gamma\in\CL_A$ such that $\rhos(\widetilde\gamma)=r$.
Now let $s\in [a,b]$ and let $(r_n)_{n\ge 0}$ be a sequence in 
$(a,b)\cap \IQ$ such that $\lim_{n\to+\infty}
r_n=s$. What precedes implies that, for all  $n\ge 0$,
there exists $\gamma_n\in\CL_A$ such that $\rhos(\widetilde\gamma_n)=r_n$.
Then, according to Lemma~\ref{lem:interval}(ii), there exists an infinite
path $\bar\alpha$ such that $\rhos(\bar\alpha)=s$ and $\bar\alpha$
starts at $A$ and passes infinitely many times at $A$, which implies that
$\bar\alpha\in\Gamma(C)$. This ends the proof of the proposition.
\end{proof}

\subsection{The rotation numbers of infinite paths not in connected components}

\begin{proposition}\label{prop:rotJ}
Let \[
\CJ:=\{\bar\alpha=(\alpha_n)_{n\geq 0}\in\Gamma(\CG)\mid \exists N(\bar\alpha)
\ge 0, \forall n\ge N(\bar\alpha),\nexists B\in\CP, \alpha_n\to B\}.
\]
Let $\bar\alpha=(\alpha_n)_{n\geq 0}\in\CJ$ and let $(A_n)_{n\ge 0}$ be
the sequence of elements of $\CP$ such $\alpha_n=A_0\ldots A_n/\sim$ for all
$n\ge 0$. Then $\rhos(\bar\alpha)$ exists and is a rational number. More
precisely, there exist $q\ge 1$, $B_0,\cdots, B_{q-1}\in\CP$ and $M\ge
N(\bar\alpha)$ such that
\[
\forall n\ge 0, \forall r\in\Lbrack 0,q-1\Rbrack, A_{M+nq+r}=B_r
\]
and $\rhos(\bar\alpha)=\frac{p}{q}$, where 
$p:=p(B_0)+p(B_1)+\cdots+p(B_{q-1})$.
Moreover, there exists $x\in X^{\infty}$ such that $F^q(x)=x+p$, that is,
$x$ is periodic $\modi$ and $\rhos[F](x)=\frac pq=\rhos(\bar\alpha)$.
\end{proposition}

\begin{proof}
Let $n\ge N(\bar\alpha)$. By Lemma~\ref{lem:one-arrow}, there exists
$i\in\Lbrack 1,\ell(A_n)\Rbrack$ such that 
$A_{n+1}= X^{\ell(A_n)}_i$ and, if $i\ge 2$, then 
$\alpha\to X^{\ell(A_n)}_1$. By definition of $\CJ$ and choice of $n$, 
there is no arrow from $\alpha$ to some element of $\CP$.
Therefore $i=1$. This implies that, for all
$n\ge  N(\bar\alpha)$, $A_{n+1}$ is uniquely determined by $A_n$. Since
$\CP$ is finite, there exist $M\ge N(\bar\alpha)$ and $q\ge 1$ such that
$A_{M+q}=A_M$. If we set $B_0\ldots B_{q-1}:=A_M\ldots A_{M+q-1}$, we get
\[
\forall n\ge 0, \forall r\in\Lbrack 0,q-1\Rbrack, A_{M+nq+r}=B_r
\]
Then, if we set $p:=p(B_0)+p(B_1)+\cdots+p(B_{q-1})$, it is clear
that $\rhos(\bar\alpha)=\frac pq$.
This proves the first part of the proposition.

For all $n\ge 1$, we set
\[
\beta_n:=\underbrace{B_0\ldots B_{q-1}\ \cdots\ B_0\ldots B_{q-1}}_{
B_0\ldots B_{q-1} \text{ repeated $n$ times}} /\sim.
\]
Then $\beta_n=A_M\ldots A_{M+nq-1}/\sim$. Since
$\langle A_M\ldots A_{M+m}\rangle \pluscover{F} 
\langle A_M\ldots A_{M+m+1}\rangle+p(A_{M+m})$ for all $m\ge 0$, 
Lemma~\ref{lem:concatenatecovering} implies that
$\langle \beta_n\rangle\pluscover{F^q}\langle \beta_{n+1}\rangle+p$
for all $n\ge 1$.
We set $a_0:=\min B_{q-1}$. Then $a_0=\min \langle\beta_n\rangle$ 
for all $n\ge 1$ by Lemma~\ref{lem:subintervalAn}(ii) and
$F(a_0)\in\TR$ by Lemma~\ref{lem:basic-partition}(iii).
Let $\ell\in\Lambda$ be such that $B_{q-1}\subset X^\ell$.
By Proposition~\ref{prop:itinerary-G}, there exists $y'\in X^{\infty}$ 
of itinerary $(A_{m})_{m\ge 0}$. Thus the itinerary of $y:=F^{M+q-1}(y')$ 
is $(A_{M+q-1+m})_{m\ge 0}$.
Let $G:=F^q-p$. It is clear that $G(\TR)\subset \TR$.
For all $n\ge 0$, $G^n(y)\in B_{q-1}\subset X^\ell$, and in particular $G^n(y)\ge a_0$ for the order in $X^\ell$.
We define inductively  a sequence of points $(a_i)_{i\geq 1}$ such
that $a_i\in [a_{i-1},y]$ and $G^i(a_i)=a_0$ for all $i\geq 1$.
\begin{itemize}
\item Since $G(a_0)\in\TR$ and $G(y)\ge a_0$, we have 
$a_0\in G([a_0,y])$ by continuity. Thus there exists $a_1\in [a_0,y]$ such that
$G(a_1)=a_0$. 
\item Assume that $a_0,\ldots, a_i$ are already defined. 
Since $G^{i+1}(a_i)=G(a_0)\in\TR$ and $G^{i+1}(y)\geq a_0$, the point $a_0$
belongs to $G^{i+1}([a_i,y])$ by continuity. Thus there exists $a_{i+1}
\in [a_i,y]$ such that $G^{i+1}(a_{i+1})=a_0$. 
This concludes the construction of $a_{i+1}$.
\end{itemize}
The sequence $(a_i)_{i\geq 0}$ is non decreasing and contained in the
compact interval $X^\ell$. Therefore $x:=\lim_{i\to+\infty}a_i$ exists
and belongs to $X^\ell$.
Since $G(a_{i+1})=a_i$, we get that $G(x)=x$.
In other words, $F^q(x)=x+p$. Thus $F^{nq}(x)\in X$ for all $n\ge 0$,
which implies that $x\in X^{\infty}$. Clearly, $x$ is periodic $\modi$
and $\rhos[F](x)=\frac pq$.
\end{proof}

\subsection{The rotation set of $\Gamma(\CG)$}

Proposition~\ref{prop:structure-paths} implies that the rotation set
of $\Gamma(\CG)$ can be decomposed as the finite union of the
rotation sets of the connected components plus a finite set
(see Equation~\eqref{eq:decomposition-RotG} in the proof below). 
This decomposition, together
with the study of the rotation set of a connected component in the previous
subsection, leads to the following theorem.

\begin{theorem}\label{theo:RotG}
The set $\Rot_{\Gamma(\CG)}$ is compact and has finitely many connected 
components.

There exists a finite set $D$ such that, for every rational number $p/q$
in $\Rot_{\Gamma(\CG)}\setminus D$, there exists a loop $\gamma$ in $\CG$
such that $\rho(\widetilde\gamma)=p/q$.
\end{theorem}

\begin{proof}
Let $\CC$ be the
set of connected components of $\CG$. Let $\CI$ and $\CJ$ be the sets
defined in Proposition \ref{prop:structure-paths} and \ref{prop:rotJ}
respectively. Let $\bar\alpha=(\alpha_n)_{n\geq 0}\in\Gamma(\CG)$ 
such that $\rhos(\bar\alpha)$ exists.
By Proposition~\ref{prop:structure-paths}(iii):

\noindent $\bullet$
Either there exist a connected component $C\in\CC$ and an integer 
$N$ such that, $\forall n\geq N, \alpha_n\in C$. In this case,
$\rho(\bar\alpha)\in \Rot_{\Gamma(C)}$.

\noindent $\bullet$
Or there exist $\bar\beta=(\beta_n)_{n\geq 0}\in\CI\cap \CJ$ and integers $N,M$
such that, $\forall n\geq 0$, $\alpha_{N+n}=\beta_{M+n}$. Thus
$\rho(\bar\beta)$ exists and is a rational number (Proposition~\ref{prop:rotJ})
and is equal to $\rho(\bar\alpha)$, and so
$\rho(\bar\alpha)\in\Rot_{\CI\cap \CJ}$.

Consequently,
\begin{equation}\label{eq:decomposition-RotG}
\Rot_{\Gamma(\CG)}=\bigcup_{C\in\CC}\Rot_{\Gamma(C)}\cup\Rot_{\CI\cap \CJ}.
\end{equation}
The sets $\CC$ and $\CI$ are finite by 
Proposition~\ref{prop:structure-paths}(i)-(ii),
and thus $\Rot_{\CI\cap \CJ}$ is finite. Moreover, for every $C\in\CC$, 
$\Rot_{\Gamma(C)}$ is a compact interval by Proposition~\ref{prop:RotC}(i).
Therefore $\Rot_{\Gamma(\CG)}$ is a finite union of compact intervals
(some intervals may be reduced to a single point). This proves the first
point of the theorem.

According to Proposition~\ref{prop:RotC}(ii), for every rational number
$p/q\in \bigcup_{C\in\CC}\Int{\Rot_{\Gamma(C)}}$, there exists a loop $\gamma$
such that $\rho(\widetilde\gamma)=p/q$. Since the set $D:=\bigcup_{C\in\CC}
\partial \Rot_{\Gamma(C)}\cup \Rot_{\CI\cap\CJ}$ is finite, this implies the
second point of the theorem.
\end{proof}

\begin{remark}
\begin{itemize}
\item The intervals $(\Rot_{\Gamma(C)})_{C\in\CC}$ 
may not be disjoint.
\item The number of connected components of $\Rot_{\Gamma(\CG)}$ is at most
$2\#\CP$ and at most $\#\CP$ are not reduced to a single point. This is because
both $\#\CC$ and $\#\CI$ are bounded by $\#\CP$.
\end{itemize}
\end{remark}

\section{The rotation set of a sun-like map}

Now that we have studied the rotation set of the covering graph, 
we are ready to prove that the rotation set of a sun-like map
is composed of finitely many compact intervals, and that all but finitely
many rational numbers in the rotation set are rotation numbers of some
periodic {\modi} points.

According to Equation~\eqref{eq:decomposition-RotF}, the rotation set
of $F$ has the following decomposition:
\[
\Rot(F)=\Rot_{\TR}(F)\cup\Rot_{X^\infty}(F).
\]

First, we pull back the results from $\Rot_{\Gamma(\CG)}$ to $\Rot_{X^\infty}(F)$.

\begin{proposition}\label{prop:RotG-RotF}
Let $F\in\sunny$ and $\CG$ its covering graph. Then
$\Rot_{X^\infty}(F)=\Rot_{\Gamma(\CG)}$. 

If $\gamma$ is a loop in $\CG$, then $\rho(\widetilde\gamma)\in\IQ$ and
there exists a periodic {\modi} point $x\in X^{\infty}$ such that 
$\rho_F(x)=\rho(\widetilde\gamma)$. If $\bar\alpha\in\CJ$, then 
$\rho(\bar\alpha)\in\IQ$ and
there exists a periodic {\modi} point $x\in X^{\infty}$ such that 
$\rho_F(x)=\rho(\bar\alpha)$
\end{proposition}

\begin{proof}
If we put together Lemma~\ref{lem:itinerary-rho}, 
Proposition~\ref{prop:itinerary-G}
and the definition of the rotation numbers in $\Gamma(\CG)$, then it is
clear that
\[
\exists x\in X^\infty\text{ such that }\rho_F(x)=r\Longleftrightarrow
\exists \bar\alpha\in\Gamma(\CG)\text{ such that }\rho(\bar\alpha)=r.
\]
Hence $\Rot_{X^\infty}(F)=\Rot_{\Gamma(\CG)}$.

Suppose that $\gamma=\alpha_0\ldots\alpha_n$ is a loop in $\CG$.
According to Lemma~\ref{lem:covering-in-graph}, we have the following
chain of positive coverings:
\[
\langle\alpha_0\rangle\pluscover{F}\langle\alpha_1\rangle+W(\alpha_0\alpha_1)
\pluscover{F}\langle\alpha_2\rangle+W(\alpha_0\alpha_1\alpha_2)
\cdots\pluscover{F}\langle\alpha_n\rangle+W(\alpha_0\ldots\alpha_n).
\]
Since $\alpha_0=\alpha_n$, Proposition~\ref{prop:+cover-periodic}
implies that there exists $x\in \langle \alpha_0\rangle$ such that
$F^n(x)=x+W(\alpha_0\ldots\alpha_n)$ and $F^i(x)\in \langle\alpha_i\rangle
+W(\alpha_0\ldots\alpha_i)$ for all $i\in\Lbrack 1, n-1\Rbrack$. 
This implies that $x$ is a periodic {\modi} point, $x\in X^\infty$ 
because $\langle\alpha_i\rangle\subset X$ for all $i\in\Lbrack 0,n\Rbrack$, and
\[
\rho_F(x)=\frac{W(\alpha_0\ldots\alpha_n)}{n}=\frac{W(\gamma)}{L(\gamma)}
=\rho(\widetilde\gamma).
\]
If $\bar\alpha\in\CJ$, the conclusion is given by Proposition~\ref{prop:rotJ}.
\end{proof}

\medskip
The main result of this paper is now a mere consequence of
Proposition~\ref{prop:RotG-RotF}, Proposition~\ref{prop:rotJ},
Theorem~\ref{theo:RotG} and Theorem~\ref{theo:RotR}.

\begin{theorem}
Let $F\in\sunny$. Then
$\Rot(F)$ is a nonempty compact set and has finitely many connected 
components. 
There exists a finite set $E$ such that, for every rational number $p/q$
in $\Rot(F)\setminus E$, there exists a periodic {\modi} point
$x\in T$  such that $\rho_F(x)=p/q$.
More precisely, $\# E\le 2\#\CP$ and $E\subset \bigcup_{C\in\CC}
\partial \Rot(\Gamma(C))$, where $\CC$ is the set of connected components of
$\CG$.
\end{theorem}

We conjecture that the exceptional set $E$ is empty. This conjecture is
true for the graph $\sigma$ \cite{R9}. In the general case of a
degree-one map on a graph with a unique loop, it is not known if the
rotation set is closed, if the number of connected components is finite or
if (almost) all rational rotation numbers are rotation numbers of periodic
points.

\bibliographystyle{plain}

\bigskip
\noindent
Laboratoire de Math\'ematiques d'Orsay,
CNRS UMR 8628, B\^atiment 307,
Universit\'e Paris-Sud 11,
91405 Orsay cedex, France. \texttt{<Sylvie.Ruette@math.u-psud.fr>}

\end{document}